\documentclass[12pt]{amsart}

\newcommand{\set}[1]{\left\{#1\right\}}
\newcommand{\abs}[1]{\left|#1\right|}
\newcommand{\brac}[1]{\left(#1\right)}

\usepackage{amssymb}

\usepackage{enumerate}

\usepackage{pgfplots}
\usepgfplotslibrary{dateplot}
\usepackage{graphicx}
\usepackage{cases}
\makeatletter
\@namedef{subjclassname@2010}{%
  \textup{2010} Mathematics Subject Classification}
\makeatother

\newtheorem{thm}{Theorem}[section]

\newtheorem{lem}[thm]{Lemma}

\newtheorem*{conj*}{Denjoy's Conjecture}

\theoremstyle{definition}

\numberwithin{equation}{section}

\begin{document}


\baselineskip=17pt



\title[LDEs with solutions in weighted Fock spaces]{Linear differential equations with solutions in weighted Fock spaces}

\author[G. M. Hu\ and J.-M. Huusko]{ Guangming Hu \ and  Juha-Matti Huusko$^{*}$}
\address{Guangming Hu \newline\ School of Mathematics and Systems Science, Beihang University,
 Beijing, 100191, P. R. China.}
 \email{18810692738@163.com}
\address{  Corresponding author$^{*}$
\newline
Department of Physics and Mathematics, University of Eastern Finland,
P.O.Box 111, FI-80101 Joensuu,  Finland }
\email{juha-matti.huusko@uef.fi}


\date{}

\begin{abstract}
This research is concerned with the nonhomogeneous  linear complex differential equation
$$
f^{(k)}+A_{k-1}f^{(k-1)}+\cdots+A_{1}f'+A_{0}f=A_{k}
$$
in the complex plane. In the higher order case, the mutual  relations between  coefficients and  solutions in weighted Fock  spaces are discussed, respectively. In particular, sufficient conditions for the solutions of  the second order case
$$
f''+Af=0
$$
to be in some weighted Fock  space are given by Bergman reproducing kernel and  coefficient $A$.
\end{abstract}



\maketitle

\section{Introduction and main results}
\indent The growth of solutions of the complex linear differential equation
$$
f^{(k)}+A_{k-1}f^{(k-1)}+\cdots+A_{1}f'+A_{0}f=A_{k}\eqno{(1)}
$$ has drawn wide attention.  Abundant results about the growth of fast growing solutions of $(1)$ have been obtained  by Nevanlinna theory \cite{a,c,b,d,e,g,j}. However, some other methods are needed in dealing with slow growth analytic solutions \cite{h,i,p,v,k,y}. There are some useful and powerful techniques, for instance, Herold's comparison theorem \cite{r}, Gronwall's lemma \cite{p}, Picard's successive approximations \cite{e,o} and some methods based on Carleson measures \cite{q,n,s,k}.\\
 \indent In \cite{i}, J. Heittokangas, R. Korhonen, and J.~R\"atty\"a found that if the coefficients $A_{j}(z)$ ($j=0,\cdots,k-1$) of the homogeneous equation $(1)$ belong to certain weighted Bergman or Hardy spaces, then all solutions are finite order of growth.  Conversely, if all solutions of  the homogeneous equation $(1)$ are  finite order of growth, the coefficients $A_{j}$ $(j=0,\cdots,k-1)$ belong to weighted Bergman space. In \cite{r},  authors  gave the growth estimates of solutions of  $(1)$.\\
\indent Motivated by the research in \cite{r} and \cite{i}, we study similar problems in weighted Fock spaces. These spaces have a long history in mathematics and mathematical physics. They have been given a wide variety of appellations, including many combinations and permutations of the names Fock, Bargmann, Segal and Fischer (see \cite{C.,C..,E.,V..}).  The classical Fock space (see \cite{V1,V2,k.zhu}) is a subspace of a special Lebesgue measurable space. Let $L_{g}^{p}$  be the space of Lebesgue measurable functions $f$ on the complex plane $\mathbb{C}$ such that the function $f(z)e^{-\frac{1}{2}|z|^{2}}$ belongs to $L^{p}(\mathbb{C},dm(z))$, for some $p\in [1,\infty]$.  The norm of $f$ in the space $L_{g}^{p}$ is defined by
$$
\|f\|_{p}^p:=\int_{\mathbb{C}}|f(z)e^{-\frac{1}{2}|z|^{2}}|^{p} dm(z),
$$
for $p\in [1,\infty)$, and
$$
\|f\|_{\infty}:=\sup_{z\in\mathbb{C}}|f(z)|e^{-\frac{1}{2}|z|^{2}},
$$for $p=\infty$, where $dm(z)=dxdy$ denotes the Lebesgue measure in $\mathbb{C}$.
The Fock space $F^{p}$ consists of entire functions in  $L_{g}^{p}$.
In particular, the space $F^{2}$ is a closed subspace of the Hilbert space $L_{g}^{2}$ with the inner product
 $$ \langle f,g\rangle: = \frac{1}{\pi}\int_{\mathbb{C}}f(z)\overline{g(z)}e^{-|z|^{2}}dm(z).$$ Recently, a new function space called Fock-Sobolev space attracted the attention of many scholars and  was first put forward by Hong Rae Cho and Kehe Zhu in \cite{zhu}.  Fock-Sobolev space $F^{p,m}$  consists of entire functions $f$ on $\mathbb{C}$ such that
 $$ \|f\|_{F^{p,m}}:=\sum_{\alpha\leq m}\|f^{(\alpha)}(z)\|_{p}<\infty,
 $$where $\|\cdot\|_{p}$ is the norm in $F^{p}$, $m$ is a positive integer, and $p\in [1,\infty]$. By Theorem $A$ in \cite{zhu}, we obtain that $f\in F^{p,m}$ if and only if $z^{m}f(z)$ is in $F^{p}$. Namely, there is a positive constant $C$ such that $$
 C^{-1}\|z^{m}f(z)\|_{p}\leq \sum_{\alpha\leq m}\|f^{(\alpha)}(z)\|_{p}\leq C\|z^{m}f(z)\|_{p}.$$
 Moreover, the weighted Fock space \cite{O.C} is defined as follows. Let $\phi: [0,\infty)\rightarrow\mathbb{R}^{+}$ be a twice continuously differentiable function and extend $\phi$ to the complex plane $\mathbb{C}$ setting $\phi(z)=\phi(|z|)$, for $z\in\mathbb{C}$. We set the norms of the weighted Fock spaces
$$
\|f\|_{F_{\phi}^{p}}^p:=\int_{\mathbb{C}}|f(z)|^{p}e^{-p\phi(z)} dm(z),
$$
for $p\in [1,\infty)$, and
$$
\|f\|_{F_{\phi}^{\infty}}:=\sup_{z\in\mathbb{C}}|f(z)|e^{-\phi(z)},
$$for $p=\infty$.
This definition gives the classical Fock space, when $\phi(z)=|z|^2/2$, and the Fock-Sobolev space, when $\phi(z)=|z|^2/2-m\log |z|$.\\
\indent For the remainder of this paper, we restrict the weight $\phi$ to be a so-called rapidly increasing function. Namely, let $\phi:[0,\infty)\rightarrow \mathbb{R}^{+}$ be twice continuously differentiable such that its Laplacian satisfies $\Delta \phi(|z|)>0$. Then there exists a positive differentiable function $\tau:\mathbb{C}\rightarrow\mathbb{R}^{+}$, with $\tau(|z|)=\tau(z)$, and a constant $C\in(0,\infty)$ such that $\tau(z)=C$, for $0\leq|z|<1$, and
$$
C^{-1} (\Delta \phi(|z|))^{-\frac{1}{2}}\leq\tau(z)\leq C(\Delta \phi(|z|))^{-\frac{1}{2}},\quad |z|\geq 1.
$$
We have $\tau(z)\to 0$ as $|z|\rightarrow\infty$ and $\lim_{r\rightarrow\infty}\tau'(r)=0$. Furthermore, suppose that either there exists a constant $C>0$ such that $\tau(r)r^{C}$ increases for large $r$ or
$$\lim_{r\rightarrow\infty}\tau'(r)\log\frac{1}{\tau(r)}=0.
$$ Denote by $\mathcal{I}$ the class of rapidly increasing functions $\phi$ satisfying the above mentioned conditions. The class $\mathcal{I}$ includes the power functions $\phi(r)=r^{\alpha}$ with $\alpha>2$ and exponential type functions such that $\phi(r)=e^{\beta r}$, $\beta>0$ or $\phi(r)=e^{e^{r}}$.

Denote the point evaluations by $L_{\zeta}(f)=f(\zeta)$, for $f\in  F^{2}_{\phi}$ and $\zeta\in\mathbb{C}$.
In \cite{O.C}, it is proved that the point evaluations $L_\zeta$ are bounded linear functionals in $F^{2}_{\phi}$. It follows that, there exists a reproducing kernel $K_{\zeta}\in F^{2}_{\phi}$, with $\|K_{\zeta}\|_{ F^{2}_{\phi}}=\|L_{\zeta}\|_{F^{2}_{\phi}\to\mathbb{C}}$, such that
$$f(\zeta)=L_{\zeta}(f(z))=\int_{\mathbb{C}}f(z)\overline{K_{\zeta}(z)}e^{-2\phi(z)}dm(z),$$for $\phi$ in the class $\mathcal{I}$. Moreover, if $\{e_{n}\}$ is an orthonormal basis of $F_{\phi}^{2}$, for example,
$$
e_{n}(z)=z^{n}\delta^{-1}_{n}, \quad n\in\mathbb{N},
$$
where $\delta_{n}^{2}=2\pi\int_{0}^{\infty}r^{2n+1}e^{-2\phi(r)}dr$, then
$$K_{\zeta}(z)=\Sigma_{n=0}^{\infty}\langle K_{\zeta}, e_{n}\rangle e_{n}(\zeta)=\Sigma_{n=0}^{\infty}e_{n}(\zeta)\overline{e_{n}(z)}.
$$

\indent By using similar idea with weighted Fock  space, Fock-Sobolev weighted space is defined  as follows. Let $\phi: [0,\infty)\rightarrow\mathbb{R}^{+}$ be a twice continuously differentiable function and extend $\phi$ to the complex plane $\mathbb{C}$ setting $\phi(z)=\phi(|z|)$, for $z\in\mathbb{C}$. We set the norms of the weighted Fock-Sobolev spaces
$$
\|f\|_{F^{p,q} _{\phi}}:=\bigg( \int_{\mathbb{C}}|f(z)|^{p}e^{-p\phi(z)}\phi^{q}(z) d(z)\bigg)^{\frac{1}{p}},
 $$
 for $p\in [1,\infty)$ and $q\in\mathbb{R}$.  The weighted Fock-Sobolev space is the weighted Fock space, when $q=0$.\\
\indent Motivated by the result of \cite{zj}, we can obtain the following two theorems.
\begin{thm}\label{1} {\normalfont} Suppose that  $A_{j}(z)$ are entire functions, $j=1,\cdots,k$. If  $$C\sum_{i=0}^{k-1}D_{i}\sup_{z\in\mathbb{C}}\bigg\{\frac{|A_{i}(z)|}
{(1+|z|)^{k-i}}\bigg\}<1$$ and the $k$th order primitive function $\varphi_{k}(z)$ of $A_{k}(z)$ belongs to $F^{p}$, where $C$ and $D_{i}$ only depend on $p,k$. Then all solutions of $(1)$ belong to $F^{p}$.
\end{thm}

\begin{thm}\label{1} {\normalfont} Suppose that  $A_{j}(z)$ are entire functions, $j=1,\cdots,k$. If    $$\sum_{i=0}^{k} C_{i}\bigg(\sum_{j=0}^{k-1}\sup_{z\in\mathbb{C}}\bigg\{\frac{|A_{j}|}{(1+|z|)^{k-i-j}}\bigg\}E_{j}\bigg)<1$$ and the $k$th order primitive  function $\varphi_{k}(z)$ of $A_{k}(z)$ belongs to $F^{p,k}$, where $E_{i}$ and $G_{j}$ only depend on $p,k$. Then all solutions of $(1)$ belong to $F^{p,k}$.
\end{thm}

Note that, unfortunately, the assumptions in Theorem~1.1 force the coefficients $A_j$ to be polynomials such that $\deg(A_j)\leq k-j$. Moreover, note that the assumption in Theorem~1.2 forces the coefficients $A_j$ to be polynomials and satisfy
$$
\deg(A_j)\leq k-i-j,\quad 0\leq i\leq k.
$$
Hence, $A_1,\ldots,A_{k-1}\equiv 0$ and $A_0\in\mathbb{C}$.

\indent  Next,  we obtain the following theorem by the study of \cite{r}.
\begin{thm}\label{1}Suppose that  $A_{k}(z)$ is not a constant function and there exists a nonconstant function among $A_{j}$, $j=0,\cdots,k-1.$ If $\phi$ is in the class $\mathcal{I}$ and there exists a sufficiently large $r_{0}$ such that $$ |A_{j}(re^{i\theta})|\leq \frac{\phi^{\frac{1}{2}}(r)}{r}
$$ for $r>r_{0}$.  Then all solutions of  $(1)$ belong to  $F^{p,q} _{\phi}.$
\end{thm}
\indent  Next, we give sufficient conditions for some coefficients to be in some weighted Fock  space.
\begin{thm}\label{1} Suppose that  $A_{j}(z)$  are constant function, $j=0,\cdots,k-1$. If there exists a solution of  $(1)$  belonging to $F^{p,k}$. Then the coefficient $A_{k}(z)$ of  $(1)$ belongs to $F^{p}$.

\end{thm}
\begin{thm}\label{1}   Let  $\phi$ be in the class $\mathcal{I}$ and moreover $\phi(r)=e^{r}$.  Suppose that $A _{j}(z)$  belongs to $F_{\frac{1}{2}e^{r}}^{p}$, $j=1,\cdots,k-1$. If there exists a solution of $(1)$ such that $$\limsup_{r\rightarrow \infty}\frac{|f(re^{i\theta})|}{e^{e^{r}}}<1$$ and $$\frac{A_{k}(re^{i\theta})}{f(re^{i\theta})}\in F_{e^{r}}^{p,q}.$$
 Then  $A_{0}(z)\in F_{e^{r}}^{p,q}$.
\end{thm}
\indent Motivated by the study of \cite{JHR} and \cite{O.C}, we  consider that sufficient conditions for the solutions of  the second order case
$$f''+Af=0\eqno{(2)}$$ to be in some weighted Fock  space are given by Bergman reproducing kernel and  coefficient $A$.

\begin{thm}\label{1} Let  $\phi$ be in the class $\mathcal{I}$, and let $|\big(\int_{0}^{z}A(\zeta)d\zeta\big)|e^{-\phi(z)}$ be bounded in $z\in\mathbb{C}$. Suppose that
$$
X_{K}(A)=\sup_{z\in \mathbb{C}}\bigg|\int_{\mathbb{C}}\bigg(\int_{0}^{z}\overline
{K'_{\zeta}(\eta)}A(\zeta)d\zeta\bigg)(1+\phi'(\eta))^{-2}e^{-\phi(\eta)}dm(\eta)\bigg|e^{-\phi(z)}<1.
$$ Then the derivative $f'$ of each solution $f$ of $(1)$ belongs to $F_{\phi}^{\infty}$.
\end{thm}
\begin{thm}\label{1}Let  $\phi$ be in the class $\mathcal{I}$, and let $|\int_{0}^{z}A(\zeta)d\zeta|e^{-\frac{1}{2}|z|^{2}}$ be bounded in $z\in\mathbb{C}$. Suppose that
$$
Y_{K}(A)=\sup_{z\in \mathbb{C}}\bigg|\int_{\mathbb{C}}\bigg(\int_{0}^{z}\overline
{K'_{\zeta}(\eta)}A(\zeta)d\zeta\bigg)(1+\phi'(\eta))^{-2}e^{-2\phi(\eta)+\frac{1}{2}|\eta|^{2}}dm(\eta)\bigg|e^{-\frac{1}{2}|z|^{2}}<1.
$$
Then the derivative $f'$ of each solutions $f$ of $(1)$ belongs to $F^{\infty}$.
\end{thm}
\begin{thm}\label{1}Let  $\phi$ be in the class $\mathcal{I}$ satisfied with Lemma $2.8$. Suppose that
$$
Z_{K}(A)=\sup_{\eta\in\mathbb{C}}\bigg|\int_{\mathbb{C}}\bigg(\int_{0}^{z}\overline{K_{\zeta}(\eta)}A(\zeta)d\zeta\bigg)\frac{e^{-2\phi(|z|)}}{(1+\phi'(|z|))^{2}}dm(z)\bigg|<1.
$$
Then all solutions $f$ of $(1)$ belong to $F_{\phi}^{2}$.
\end{thm}
\section{Preliminary Lemmas }
  \begin{lem}\label{1} {\normalfont \cite{i.laine}} Suppose that $A_{j}(z)$ are entire functions, $j=0,\cdots,k$. Then all solutions of $(1)$ are entire functions.
  \end{lem}

  \begin{lem}\label{1} {\normalfont \cite{zj}} Suppose that  $0<p<\infty$  and  $ m$ is a positive integer. Then for any entire function $f(z)$,
   $$C^{-1} \|f\|_{p}\leq\sum_{\alpha\leq m-1}|f^{(\alpha)}(0)|+\bigg(\int_{\mathbb{C}}|f^{(m)}(z)\frac{1}{(1+|z|)^{m}}e^{-\frac{1}{2}|z|^{2}}|^{p}dm(z)\bigg)^{\frac{1}{p}}\leq C \|f\|_{p},$$ where $C$ only depends on $p,m$.
  \end{lem}

  \begin{lem}\label{1} {\normalfont \cite{r}} Suppose that $f(z)$ is a solution of  $(1)$ in the disk $\Delta_{R}=\{z\in\mathbb{C}: |z|<R\}$, where $0<R\leq\infty$. Let $0\leq k_{c}\leq k$ be the number of nonzero coefficients $A_{j}(z)$ $(j=0,\cdots, k-1)$. Let $R_{0}$ be a positive real number  such that there exists some $A_{j}(R_{0}e^{i\theta})\neq 0$. Then for all $R_{0}<r<R$, there exists
 $$ |f(re^{i\theta})|\leq C \left( \max_{0\leq x\leq r}|A_{k}(x e^{i\theta})|+1\right)\exp\int_{0}^{r}\left(\delta+k_{c}\max_{0\leq j\leq k-1}|A_{j}(se^{i\theta})|^{\frac{1}{k-j}}\right)ds,
 $$where $C$ is some positive constant depending on the values of the derivatives of $f$ and the values of $A_{j}$ at $R_{0}e^{i\theta}$, $j=0,\cdots,k$. Here
 \begin{numcases}{\delta=}
0, & if $A_{k} \equiv 0,$\nonumber\\
1, &  \textrm{otherwise}\nonumber.
\end{numcases}
\end{lem}
 \begin{lem}\label{1} {\normalfont \cite{GG}} Suppose that  $f$ is a transcendental meromorphic function and $\alpha$ is a positive constant. Then there exist a set $W\subset [0,2\pi]$ of linear measure zero, a constant $C>0$ of depending only on $\alpha$ and a constant $r_{0}=r_{0}(\theta)>1$
 such that
 $$\bigg|\frac{f^{(m)}(re^{i\theta})}{f(re^{i\theta})}\bigg|\leq C\bigg(T(\alpha r,f)\frac{\log^{\alpha}r}{r}\log T(\alpha r,f)\bigg)^{m},\quad m\in\mathbb{N},
 $$ where $r>r_{0}$ and $\theta\in[0,2\pi)\backslash W $.
 \end{lem}
 \begin{lem}\label{1}  Suppose that $\phi$ is in the class $\mathcal{I}$ and $f(z)$ is  an entire function.  Then for any $R>0$, there is a constant $C>0$ such that
 $$ \int_{\mathbb{C}}|f(z)|^{p}e^{-p\phi(z)}\phi^{q}(z) dm(z)\leq C\int_{|z|\geq R}|f(z)|^{p}e^{-p\phi(z)}\phi^{q}(z) dm(z).
 $$
 \end{lem}
 \begin{proof} $$\int_{R\leq|z|\leq R+1}|f(z)|^{p}e^{-p\phi(z)}\phi^{q}(z) dm(z) =\int_{R}^{R+1}\int_{0}^{2\pi}|f(re^{i\theta})|^{p}e^{-p\phi(r)}\phi^{q}(r)rdrd\theta
$$$$\geq\min_{R\leq r\leq R+1}e^{-p\phi(r)}\phi^{q}(r)r\int_{R}^{R+1}\int_{0}^{2\pi}|f^{p}(re^{i\theta})|d\theta dr.$$Since $f^{p}(z)$ is an entire function, $|f^{p}(z)|$ is a subharmonic function. If $$\int_{0}^{2\pi}|f^{p}(re^{i\theta})|d\theta =0,$$ then $f(z)\equiv0$ and the inequality holds.  If $\int_{0}^{2\pi}|f^{p}(re^{i\theta})|d\theta >0$, then there exists $C_{1}>0$ such that
$$\int_{|z|\leq R}|f(z)|^{p}e^{-p\phi(z)}\phi^{q}(z) dm(z)\leq C_{1}\int_{R\leq|z|\leq R+1}|f(z)|^{p}e^{-p\phi(z)}\phi^{q}(z) dm(z).
$$ Therefore, $$ \int_{\mathbb{C}}|f(z)|^{p}e^{-p\phi(z)}\phi^{q}(z) dm(z)\leq C\int_{|z|\geq R}|f(z)|^{p}e^{-p\phi(z)}\phi^{q}(z) dm(z),
 $$ where $C=C_{1}+1$.
 \end{proof}
 \begin{lem}\label{1}  Suppose that $\phi$ is in the class $\mathcal{I}$. Then $$\lim_{r\rightarrow\infty}\frac{\phi(r)}{r^{2}}=\infty.$$
 \end{lem}

 \begin{proof} By L'Hospital's rule, we obtain $$\lim_{r\rightarrow\infty}\frac{\phi(r)}{r^{2}}=\lim_{r\rightarrow\infty}\frac{\phi'(r)}{2r}=\lim_{r\rightarrow\infty}\frac{r\phi'(r)}{2r^{2}}
 =\lim_{r\rightarrow\infty}\frac{(r\phi'(r))'}{4r}$$
 $$=\frac{1}{4}(\phi''(r)+\frac{\phi'(r)}{r})=\frac{1}{4}\Delta\phi(r).
 $$ Since $C^{-1} (\Delta \phi(|z|))^{-\frac{1}{2}}\leq\tau(z)$ and $\tau(z)$ decreases to $0$ as $|z|\rightarrow\infty$. Then $$\lim_{r\rightarrow\infty}\frac{\phi(r)}{r^{2}}=\infty.$$
 \end{proof}
  \begin{lem}\label{1} \cite{O.C} Suppose that $\phi$ is in the class $\mathcal{I}$ and $f,g\in F_{\phi}^{2}$. Then
 $$ \langle f,g\rangle=f(0)\overline{g(0)}+\int_{\mathbb{C}}f'(z)\overline{g'(z)}(1+\phi'(z))^{-2}e^{-2\phi(z)}dm(z)
 .$$
 \end{lem}
  \begin{lem}\label{1} \cite{O.C} Suppose that $\phi$ is in the class $\mathcal{I}$ and there exists $r_{0}>0$
such that $\phi'(r)\neq0$ for $r>r_{0}$. Moreover, assume that $\phi$ satisfies
$$
\lim_{r\rightarrow \infty}\frac{re^{-p\phi(r)}}{\phi'(r)}=0,
$$
and
$$
-\infty
<\liminf_{r\rightarrow\infty}\frac{1}{r}\bigg( \frac{r}{\phi'(r)}\bigg)'
\leq \limsup_{r\rightarrow\infty}\frac{1}{r}\bigg( \frac{r}{\phi'(r)}\bigg)'<p,
$$
where $p\geq1$. Then for any entire function $f(z)$,
   $$C^{-1} \|f\|^{p}_{F_{\phi}^{p}}\leq|f(0)|^{p}+\int_{\mathbb{C}}|f'(z)|\frac{e^{-p\phi(|z|)}}{(1+\phi'(|z|))^{p}}dm(z)\leq C \|f\|^{p}_{F_{\phi}^{p}},$$ where $C$ only depends on $p$.
 \end{lem}

\section{Proof of Theorems $1.1$ and Theorems $1.2$}

\begin{proof}[\bf Proof of Theorem $1.1$] Suppose that $f$ is a solution of $(1)$, then  $f(z)$ is an entire function by Lemma $2.1$.  By Lemma $2.2$,
\begin{equation*}
\begin{aligned}
\|f\|_{p}&\leq C\bigg( \sum_{\alpha\leq k-1}|f^{(\alpha)}(0)|+\bigg(\int_{\mathbb{C}}\bigg|f^{(k)}(z)\frac{1}{(1+|z|)^{k}}e^{-\frac{1}{2}|z|^{2}}\bigg|^{p}dm(z)\bigg)^{\frac{1}{p}}\bigg)\\
&=C_{1}+C\bigg(\int_{\mathbb{C}}\left|f^{(k)}(z)\frac{1}{(1+|z|)^{k}}e^{-\frac{1}{2}|z|^{2}}\right|^{p}dm(z)\bigg)^{\frac{1}{p}},\\
\end{aligned}
\end{equation*}
 where $C_{1}=C\sum_{\alpha\leq k-1}|f^{(\alpha)}(0)|.$ Using $(1)$ and Minkowski inequality,
  \begin{equation*}
\begin{aligned}
\|f\|_{p}
&\leq C_{1}+C\brac{\int_{\mathbb{C}}\bigg|\bigg(\sum_{i=0}^{k-1}A_i(z)f^{(i)}(z)-A_k(z)\bigg)\frac{e^{-\frac{1}{2}|z|^{2}}}{(1+|z|)^{k}}\bigg|^{p}dm(z)}^{\frac{1}{p}}\\
&\leq C_{1}+C\brac{\sum_{i=0}^{k-1}\brac{\int_{\mathbb{C}}\bigg|A_i(z)f^{(i)}(z)\frac{e^{-\frac{1}{2}|z|^{2}}}
{(1+|z|)^{k}}\bigg|^{p}dm(z)}^{\frac{1}{p}}}\\
&\quad +\brac{\int_{\mathbb{C}}\abs{\frac{A_{k}(z)e^{-\frac{1}{2}|z|^{2}}}{(1+|z|)^{k}}}^{p}dm(z)}^{\frac{1}{p}},\\
&\leq C_{1}+C\bigg(\sum_{i=0}^{k-1}
\sup_{z\in\mathbb{C}}\bigg\{\frac{|A_{i}(z)|}{(1+|z|)^{k-i}}\bigg\}
\bigg(\int_{\mathbb{C}}
\bigg|f^{(i)}(z)\frac{e^{-\frac{1}{2}|z|^{2}}}{(1+|z|)^{i}}\bigg|^{p}
dm(z)\bigg)^{\frac{1}{p}}\\
&\quad +(\int_{\mathbb{C}}\bigg|\frac{A_{k}(z)e^{-\frac{1}{2}|z|^{2}}}{(1+|z|)^{k}}\bigg|^{p}dm(z))^{\frac{1}{p}}\bigg).
\end{aligned}
\end{equation*}
  Using Lemma $2.2$ again,
 \begin{equation*}
\begin{aligned}
\|f(z)\|_{p}&\leq  C_{1}+C\bigg(\sum_{i=0}^{k-1}\sup_{z\in\mathbb{C}}\{\frac{|A_{i}|}{(1+|z|)^{k-i}}\}D_{i}\|f\|_{p}+
D_{k}\|\varphi_{k}\|_{p}\bigg),
\end{aligned}
\end{equation*}
where  $\varphi_{k}(z)$ is the $k$th primitive function of $A_{k}(z)$. Therefore
$$
\|f\|_{p}\brac{1-C\brac{\sum_{i=0}^{k-1}\sup_{z\in\mathbb{C}}\set{\frac{|A_{i}|}{(1+|z|)^{k-i}}}D_{i}}}\leq  C_{1}+
D_{k}\|\varphi_{k}\|_{p}
$$
If $\|f\|_{p}=\infty$, it  contradicts  the condition of Theorem $1.1$. Therefore,  $$
\|f(z)\|_{p}\leq\frac{C_{1}+CD_{k}\|\varphi_{k}\|_{p}}{1-C\sum_{i=0}^{k-1}D_{i}\sup_{z\in\mathbb{C}}\{\frac{|A_{i}|}
{(1+|z|)^{k-i}}\}}<+\infty$$  and $f\in F^{p}$.
\end{proof}

\begin{proof}[\bf Proof of Theorem $1.2$] Suppose that $f(z)$ is a solution of $(1)$, then  $f(z)$ is an entire function by Lemma $2.1$. Using Lemma $2.2$, $(1)$ and  Minkowski inequality,
\begin{equation*}
\begin{aligned}
\|f(z)\|_{F^{p,k}}&\leq \sum_{i=0}^{k} C_{i}\bigg( \sum_{\alpha=i}^{k-1}|f^{(\alpha)}(0)|+(\int_{\mathbb{C}}|f^{(k)}(z)\frac{1}{(1+|z|)^{k-i}}
e^{-\frac{1}{2}|z|^{2}}|^{p}dm(z))^{\frac{1}{p}}\bigg)\\&=D_{1}+\sum_{i=0}^{k} C_{i}\bigg(\int_{\mathbb{C}}|(A_{k-1}f^{(k-1)}+\cdots+A_{0}(z)f-A_{k})\frac{e^{-\frac{1}{2}|z|^{2}}}{(1+|z|)^{k-i}}|^{p}dm(z)\bigg)^{\frac{1}{p}}\\&\leq D_{1}+\sum_{i=0}^{k} C_{i}\bigg(\sum_{j=0}^{k-1}\int_{\mathbb{C}}|A_{j}f^{(j)}\frac{e^{-\frac{1}{2}|z|^{2}}}{(1+|z|)^{k-i}}|^{p}dm(z)\bigg)^{\frac{1}{p}}\\&+\sum_{i=0}^{k} C_{i}\bigg(\int_{\mathbb{C}}|\frac{A_{k}e^{-\frac{1}{2}|z|^{2}}}{(1+|z|)^{k-i}}|^{p}dm(z)\bigg)^{\frac{1}{p}}\\& \leq D_{1}+\sum_{i=0}^{k} C_{i}\sum_{j=0}^{k-1}\sup_{z\in\mathbb{C}}\bigg\{\frac{|A_{j}|}{(1+|z|)^{k-i-j}}\bigg\}
\bigg(\int_{\mathbb{C}}|f^{(j)}\frac{e^{-\frac{1}{2}|z|^{2}}}{(1+|z|)^{j}}|^{p}dm(z)\bigg)^{\frac{1}{p}}\\&+\sum_{i=0}^{k} C_{i}\bigg(\int_{\mathbb{C}}\bigg|\frac{A_{k}e^{-\frac{1}{2}|z|^{2}}}{(1+|z|)^{k-i}}\bigg|^{p}dm(z)\bigg)^{\frac{1}{p}},
\end{aligned}
\end{equation*}where $D_{1}=\sum_{i=0}^{k} C_{i} \sum_{\alpha=i}^{k-1}|f^{(\alpha)}(0)|$.
Using Lemma $2.2$ again,
 \begin{equation*}
\begin{aligned}
\|f(z)\|_{F^{p,k}}&\leq D_{1}+\sum_{i=0}^{k} C_{i}\sum_{j=0}^{k-1}\sup_{z\in\mathbb{C}}E_{j}\bigg\{\frac{|A_{j}|}{(1+|z|)^{k-i-j}}\bigg\}\|f\|_{p}+\sum_{i=0}^{k} C_{i}F_{i}\|\varphi_{k}^{(i)}\|_{p}\\&\leq D_{1}+\sum_{i=0}^{k} C_{i}\bigg(\sum_{j=0}^{k-1}E_{j}\sup_{z\in\mathbb{C}}\bigg\{\frac{|A_{j}|}{(1+|z|)^{k-i-j}}\bigg
\}\bigg)\|f\|_{F^{p,k}}+\sum_{i=0}^{k} C_{i}F_{i}\|\varphi_{k}^{(i)}\|_{p}\\&\leq D_{1}+\sum_{i=0}^{k} C_{i}\bigg(\sum_{j=0}^{k-1}\sup_{z\in\mathbb{C}}E_{j}\bigg\{\frac{|A_{j}|}{(1+|z|)^{k-i-j}}\bigg
\}\bigg)\|f\|_{F^{p,k}}+ G\|\varphi_{k}\|_{F^{p,k}}
\end{aligned}
\end{equation*}
 where $G=\max_{0\leq i\leq k}\{C_{i}F_{i}\}$.So$$\|f(z)\|_{F^{p,k}}\bigg(1-  \sum_{i=0}^{k} C_{i}\big(\sum_{j=0}^{k-1}\sup_{z\in\mathbb{C}}E_{j}\bigg\{\frac{|A_{j}|}{(1+|z|)^{k-i-j}}\bigg
\}\big) \bigg)\leq D_{1}+ G\|\varphi_{k}\|_{F^{p,k}}.$$If $\|f\|_{ F^{p,k}}=\infty$, it  contradicts  the condition of Theorem $1.2$. Therefore,
   $$
\|f(z)\|_{F^{p,k}}\leq\frac{D_{1}+G\|\varphi_{k}\|_{F^{p,k}}}{1-\sum_{i=0}^{k} C_{i}\big(\sum_{j=0}^{k-1}E_{j}\sup_{z\in\mathbb{C}}\{\frac{|A_{j}|}{(1+|z|)^{k-i-j}}\}\big)}<+\infty$$ and
  $f\in F^{p,k}$.
\end{proof}
 \section{ Proof of Theorem $1.3$ }
 \begin{proof} [\bf Proof of Theorem $1.3$]   If $f(z)$ is a solution of  $(1)$, from Lemma $2.3$,  $$|f(re^{i\theta})|\leq C  ( \max_{0\leq x\leq r}|A_{k}(x e^{i\theta})|+1)\exp\int_{0}^{r}(\delta+k_{c}\max_{0\leq j\leq k-1}|A_{j}(se^{i\theta})|^{\frac{1}{k-j}})ds.$$ Since $A_{k}(z)$ is not a constant function and there exists a nonconstant function among $A_{j}(z)$, $j=0,\cdots,k-1$,
  $$ |f(re^{i\theta})|\leq C  ( 2\max_{0\leq x\leq r}|A_{k}(x e^{i\theta})|)\exp\int_{0}^{r}(2 k_{c}\max_{0\leq j\leq k-1}|A_{j}(se^{i\theta})|^{\frac{1}{k-j}})ds
 $$for sufficiently large $r>r_{1}$. Since $|A_{j}(re^{i\theta})|\leq \phi^{\frac{1}{2}}(r)/r$, $r>r_{0}$, we have $$ |f(re^{i\theta})|\leq2C D \frac{\phi^{\frac{1}{2}}(r)}{r}\exp\bigg ( 2k_{c}\max_{0\leq j\leq k-1}\int_{R}^{r}(\frac{\phi^{\frac{1}{2}}(s)}{s})^{\frac{1}{k-j}}ds\bigg)$$ for $r>R=\max\{r_{0},r_{1},1\}$, where $$D=\exp \int_{0}^{R}(2 k_{c}\max_{0\leq j\leq k-1}|A_{j}(se^{i\theta})|^{\frac{1}{k-j}})ds.$$
Since
$$
\max_{0\leq j\leq k-1}
\int_{R}^{r}\brac{\frac{\phi^{\frac{1}{2}}(s)}{s}}^{\frac{1}{k-j}}ds
\leq\phi^{\frac{1}{2}}(r)\int_R^r s^{-\frac{1}{k}}
\leq\phi^{\frac{1}{2}}(r)\frac{k}{k-1}r^{1-\frac{1}{k}},
$$
we have
$$
|f(re^{i\theta})|
\leq
D_1\frac{\phi^{\frac{1}{2}}(r)}{r}\exp\brac{D_2\phi^{\frac{1}{2}}(r)r^{1-\frac{1}{k}}},
\quad R<r<\infty,
$$
where $D_{1}=2CD$ and $D_{2}=2k_{c}\frac{k}{k-1}$. Since $\phi$ is in the class $\mathcal{I}$ and Lemma $2.6$, there exists $R'>0$ such that $\phi(r)>r^{2}$ for $r>R'$.
Since $f(z)$ is  an entire function, by using Lemma $2.5$,

\begin{equation*}
\begin{aligned}
\| f\|^{p}_{F_{\phi}^{p,q}}&= \int_{\mathbb{C}}|f(z)|^{p}e^{-p\phi(z)}\phi^{q}(z) dm(z)\\
&\leq M\int_{|z|\geq R_{1}}|f(z)|^{p}e^{-p\phi(z)}\phi^{q}(z) dm(z)\\
  &\leq M\int_{0}^{2\pi}\int_{R_{1}}^{\infty}\bigg(D_{1}\frac{\phi^{\frac{1}{2}}(r)}{r}\exp \big( D_{2}\phi^{\frac{1}{2}}(r)r^{1-\frac{1}{k}}\big)\bigg)^{p}e^{-p\phi(r)}\phi^{q}(r)rdrd\theta\\&
\leq MD^{p}_{1}\int_{0}^{2\pi}\int_{R_{1}}^{\infty}\bigg(\frac{\phi^{\frac{p}{2}+q}(r)}{r^{p-1}}\exp\big(p(D_{2}\phi^{\frac{1}{2}}(r)r^{1-\frac{1}{k}}-\phi(r))\big)\bigg)drd\theta\\& \leq 2\pi MD^{p}_{1}\int_{R_{1}}^{\infty}\bigg(\frac{\phi^{\frac{p}{2}+q}(r)}{r^{p-1}}e^{p(D_{2}\phi^{\frac{1}{2}}(r)r^{1-\frac{1}{k}}-\phi(r))}\bigg)dr,
 \end{aligned}
\end{equation*}
 where $R_{1}=\max\{R,R'\}$. Since $D_{2}r^{1-\frac{1}{k}}-\phi^{\frac{1}{2}}(r)<-1$, when $r>R_{1}$. Thus  $$\| f\|^{p}_{F_{\phi}^{p,q}}\leq 2\pi MD^{p}_{1}\int_{R}^{\infty}\bigg(\frac{\phi^{\frac{p}{2}+q}(r)}{r^{p-1}}e^{-p\phi^{\frac{1}{2}}(r)}\bigg)dr<\infty.$$ Therefore,  $f\in F^{p,q} _{\phi}$.
 \end{proof}
 \section{Proof of Theorems $1.4$ and Theorem $1.5$}

 \begin{proof} [\bf Proof of Theorem $1.4$] By $(1)$,
 \begin{equation*}
\begin{aligned}
\|A_{k}\|_{p}&= \bigg(\int_{\mathbb{C}}|(f^{(k)}(z)+A_{k-1}(z)f^{(k-1)}(z)+\cdots+A_{0}(z)f(z))e^{-\frac{1}{2}|z|^{2}}|^{p}dm(z)\bigg)^{\frac{1}{p}}\\&\leq\bigg(\int_{\mathbb{C}}|f^{(k)}e^{-\frac{1}{2}|z|^{2}}|^{p}dm(z)\bigg)^{\frac{1}{p}}+\cdots+\bigg(\int_{\mathbb{C}}|A_{0}fe^{-\frac{1}{2}|z|^{2}}|^{p}dm(z)\bigg)^{\frac{1}{p}}\\&\leq \|f^{(k)}\|_{p}+|A_{k-1}| \|f^{(k-1)}\|_{p}+\cdots+|A_{0}|\|f\|_{p}\\&\leq C(\|f^{(k)}\|_{p}+ \|f^{(k-1)}\|_{p}+\cdots+\|f\|_{p}),
\end{aligned}
\end{equation*}where $C=\max_{i\leq k-1}\{|A_{i}|\}$.
 Since there exists a solution $f\in F^{p,k}$, we get $A_{k}\in F^{p}.$
 \end{proof}

 \begin{proof} [\bf Proof of Theorem $1.5$]
 Using $(1)$, we obtain
 $$
 |A_{0}(z)|\leq\bigg|\frac{A_{k}}{f}\bigg|+\bigg|\frac{f^{(k)}}{f}\bigg|
 +\cdots+\bigg|A_{1}\frac{f'}{f}\bigg|
 .$$
 By Lemma $2.4$, for $\alpha>1$, there exist sufficient large $r_{1}$ and a set $W$ of measure zero such that $$\bigg|\frac{f^{(j)}(re^{i\theta})}{f(re^{i\theta})}\bigg|\leq C\bigg(T(\alpha r,f)\frac{\log^{\alpha}r}{r}\log T(\alpha r,f)\bigg)^{j}, \quad j=1,\cdots,k ,$$for $r>r_{1}$ and $\theta\in [0,2\pi)\backslash W$, where $T(\alpha r,f)=\log^{+}M(\alpha r,f)$ and $M(\alpha r,f)$ is the maximum of $|f(\alpha re^{i\theta})|$ at $\theta\in [0,2\pi)$.
 By the condition of Theorem $1.5$, there exists a solution $f(z)$ such that for sufficient large $r_{2}$,
$$
T(\alpha r,f)=\log^{+}M(\alpha r,f)\leq\log{e^{e^{\alpha r}}}=e^{\alpha r},
$$
for $ r>r_{2}.$ Thus
\begin{equation*}
\begin{aligned}
|A_{0}(re^{i\theta})|&\leq\left|\frac{A_{k}(re^{i\theta})}{f(re^{i\theta})}\right|+C\bigg[(e^{\alpha r}\frac{\log^{\alpha}r}{r}\alpha r)^{k}+\cdots+|A_{1}(re^{i\theta})|e^{\alpha r}\frac{\log^{\alpha}r}{r}\alpha r\bigg],
\end{aligned}
\end{equation*}
 for $r>R=\max\{r_{1},r_{2}\}$ and $\theta\in [0,2\pi)\backslash W$. Since $A_{0}(z)$ is an entire function, using Lemma $2.5$,
  \begin{equation*}
\begin{aligned}
\|A_{0}\|^{p}_{F^{p,q}_{e^{r}}}&=\int_{\mathbb{C}}|A_{0}(z)|^{p}e^{-pe^{|z|}}(z)e^{q|z|}dm(z)\\&
\leq D\int_{0}^{2\pi}\int_{R}^{\infty}|A_{0}(re^{i\theta})|^{p}e^{-pe^{r}}(z)e^{qr}rdr d\theta.
\end{aligned}
\end{equation*}
It follows that
 \begin{equation*}
\begin{aligned}
\|A_{0}\|^{p}_{F^{p,q}_{e^{r}}}&\leq D\int_{0}^{2\pi}\int_{R}^{\infty}|\frac{A_{k}(re^{i\theta})}{f(re^{i\theta})}|^{p}e^{-pe^{r}}(z)e^{qr}rdr d\theta+D_{k}\int_{0}^{2\pi}\int_{R}^{\infty}
\frac{(\alpha e^{\alpha r}\log^{\alpha}r)^{pk}e^{qr}}{e^{pe^{r}}}rdr\\&+\cdots
 +D_{1}\int_{0}^{2\pi}\int_{R}^{\infty}
 |A_{1}(re^{i\theta})|^{p}\frac{(\alpha e^{\alpha r}\log^{\alpha}r)^{p}e^{qr}}{e^{pe^{r}}}rdrd\theta.
\end{aligned}
\end{equation*}
Since $p\geq1$, $$\int_{0}^{2\pi}\int_{R}^{\infty}
\frac{(\alpha e^{\alpha r}\log^{\alpha}r)^{pk}e^{qr}}{e^{pe^{r}}}rdr<\infty.$$
 Since $\frac{A_{k}(re^{i\theta})}{f(re^{i\theta})}\in F_{e^{r}}^{p,q}$,
  $$\int_{0}^{2\pi}\int_{R}^{\infty}|\frac{A_{k}(re^{i\theta})}{f(re^{i\theta})}|^{p}e^{-pe^{r}}(z)e^{qr}rdr d\theta<\infty.$$
Since $A _{j}(z)$  belong to $F_{\frac{1}{2}e^{r}}^{p}$, $j=1,\cdots,k-1$,
$$\int_{0}^{2\pi}\int_{R}^{\infty}|A_{j}(re^{i\theta})|^{p}\frac{(\alpha e^{\alpha r}\log^{\alpha}r)^{jp}e^{qr}}{e^{pe^{r}}}rdrd\theta\leq$$
$$\int_{0}^{2\pi}\int_{R}^{\infty}|A_{j}(re^{i\theta})|^{p}e^{-\frac{p}{2}e^{r}}rdrd\theta<\infty.$$
Therefore, $A_{0}(z)\in F_{e^{r}}^{p,q}$.
 \end{proof}
 \section{Proof of Theorems $1.6$ and Theorem $1.7$}
 \begin{proof} [\bf Proof of Theorem $1.6$] Let $f$ be any solution of $(2)$, then
 $$ f'(z)=-\int_{0}^{z}f(\zeta)A(\zeta)d\zeta+f'(0), \quad z\in\mathbb{C}.
 $$If $g$ satisfies the reproducing formula, $g\in F_{\phi}^{2}.$ Since $f$ is an entire function, there exists a finite Taylor expansion $f(z)=\sum_{j=0}^{\infty}a_{j}z^{j}$ such that $f_{n}=\sum_{j=0}^{n}a_{j}z^{j}\in F_{\phi}^{2}$. Therefore, $$ f'(z)=-\int_{0}^{z}\lim_{n\rightarrow\infty}f_{n}(\zeta)A(\zeta)d\zeta+f'(0), \quad z\in\mathbb{C}.
 $$
 By  the reproducing formula and Fubini's theorem, for $\phi$ in the class $\mathcal{I}$, we get
 \begin{equation*}
\begin{aligned}
 f'(z)&=-\int_{0}^{z}\big( \lim_{n\rightarrow\infty}\int_{\mathbb{C}}f_{n}(\eta)\overline{K_{\zeta}(\eta)}e^{-2\phi(\eta)}dm(\eta)\big)A(\zeta)d\zeta+f'(0)\\&
 =-\int_{\mathbb{C}} \lim_{n\rightarrow\infty}f_{n}(\eta)e^{-2\phi(\eta)}\big(\int_{0}^{z}\overline{K_{\zeta}(\eta)}A(\zeta)d\zeta\big)dm(\eta)+f'(0),\quad z\in\mathbb{C}.
 \end{aligned}
\end{equation*}
 Since $K_{\zeta}(0)=\Sigma_{n=0}^{\infty}e_{n}(\zeta)\overline{e_{n}(0)}=\delta_{0}^{-2}$, and using Lemma $2.7$, we get  \begin{equation*}
\begin{aligned}
 f'(z)&=
-\int_{\mathbb{C}} \lim_{n\rightarrow\infty}f_{n}'(\eta)\big(\int_{0}^{z}\overline{K'_{\zeta}(\eta)}A(\zeta)d\zeta\big)(1+\phi'(\eta))^{-2}e^{-2\phi(\eta)}dm(\eta)-\\&
 \lim_{n\rightarrow\infty}f_{n}(0)\big(\int_{0}^{z}\overline{K_{\zeta}(0)}A(\zeta)d\zeta\big)+f'(0)\\&=-\int_{\mathbb{C}}
 f'(\eta)\big(\int_{0}^{z}\overline{K'_{\zeta}(\eta)}A(\zeta)d\zeta\big)(1+\phi'(\eta))^{-2}e^{-2\phi(\eta)}dm(\eta)-\\&
 f(0)\big(\int_{0}^{z}\delta_{0}^{-2}A(\zeta)d\zeta\big)+f'(0),\quad z\in\mathbb{C}.
 \end{aligned}
\end{equation*}
It follows that\begin{equation*}
\begin{aligned}
|f'(z)|e^{-\phi(z)}&\leq
\sup_{\eta\in \mathbb{C}}\big\{|f'(\eta)|e^{-\phi(\eta)}\big\}|\int_{\mathbb{C}}\big(\int_{0}^{z}\overline{K'_{\zeta}(\eta)}A(\zeta)d\zeta\big)
(1+\phi'(\eta))^{-2}e^{-\phi(\eta)}dm(\eta)|e^{-\phi(z)}\\&
 +|f(0)\big(\int_{0}^{z}\delta_{0}^{-2}A(\zeta)d\zeta\big)|e^{-\phi(z)}+|f'(0)|,\quad z\in\mathbb{C}.
 \end{aligned}
\end{equation*}

Then, we get
\begin{equation*}
\begin{aligned}
\|f'\|_{F_{\phi}^{\infty}}&\leq
\|f'\|_{F_{\phi}^{\infty}}\sup_{z\in \mathbb{C}}\big\{|\int_{\mathbb{C}}\big(\int_{0}^{z}\overline
{K'_{\zeta}(\eta)}A(\zeta)d\zeta\big)(1+\phi'(\eta))^{-2}e^{-\phi(\eta)}dm(\eta)|e^{-\phi(z)}\big\}\\&
 +\sup_{z\in \mathbb{C}}\big\{|f(0)\big(\int_{0}^{z}\delta_{0}^{-2}A(\zeta)d\zeta\big)|e^{-\phi(z)}\big\}+|f'(0)|
 \end{aligned}
\end{equation*}
Thus $$
\|f'\|_{F_{\phi}^{\infty}}\bigg(1-\sup_{z\in \mathbb{C}}\big\{|\int_{\mathbb{C}}\big(\int_{0}^{z}\overline
{K'_{\zeta}(\eta)}A(\zeta)d\zeta\big)(1+\phi'(\eta))^{-2}e^{-\phi(\eta)}dm(\eta)|e^{-\phi(z)}\big\}\bigg)
 $$$$\leq\sup_{z\in \mathbb{C}}\big\{|f(0)\big(\int_{0}^{z}\delta_{0}^{-2}A(\zeta)d\zeta\big)|e^{-\phi(z)}\big\}+|f'(0)|
$$
If $\|f'\|_{F_{\phi}^{\infty}}=\infty$, it  contradicts  the condition of Theorem $1.6$. Therefore, $$
\|f'\|_{F_{\phi}^{\infty}}\leq\frac{1}{1-X_{K}(A)}\bigg(\sup_{z\in \mathbb{C}}\big\{|f(0)\big(\int_{0}^{z}\delta_{0}^{-2}A(\zeta)d\zeta\big)|e^{-\phi(z)}\big\}+|f'(0)|\bigg)<\infty
$$ and $f'\in F_{\phi}^{\infty}.$
 \end{proof}
 \begin{proof} [\bf Proof of Theorem $1.7$] By  the proof of Theorem $1.6$, we get   \begin{equation*}
\begin{aligned}
 f'(z)&=-\int_{\mathbb{C}}
 f'(\eta)\bigg(\int_{0}^{z}\overline{K'_{\zeta}(\eta)}A(\zeta)d\zeta\bigg)(1+\phi'(\eta))^{-2}e^{-2\phi(\eta)}dm(\eta)-\\&
 f(0)\big(\int_{0}^{z}\delta_{0}^{-2}A(\zeta)d\zeta\big)+f'(0),\quad z\in\mathbb{C}.
 \end{aligned}
\end{equation*}
It follows that\begin{equation*}
\begin{aligned}
|f'(z)|e^{-\frac{1}{2}|z|^{2}}&\leq
\|f'\|_{\infty}|\int_{\mathbb{C}}\big(\int_{0}^{z}\overline{K'_{\zeta}(\eta)}A(\zeta)d\zeta\big)(1+\phi'(\eta))^{-2}
e^{-2\phi(\eta)+\frac{1}{2}|\eta|^{2}}dxdy|e^{-\frac{1}{2}|z|^{2}}\\&
 +|f(0)\big(\int_{0}^{z}\delta_{0}^{-2}A(\zeta)d\zeta\big)|e^{-\frac{1}{2}|z|^{2}}+|f'(0)|,\quad z\in\mathbb{C}.
 \end{aligned}
\end{equation*}

Then, we get
\begin{equation*}
\begin{aligned}
\|f'\|_{\infty}&\leq
\|f'\|_{\infty}\sup_{z\in \mathbb{C}}\bigg\{|\int_{\mathbb{C}}\big(\int_{0}^{z}\overline
{K'_{\zeta}(\eta)}A(\zeta)d\zeta\big)(1+\phi'(\eta))^{-2}e^{-2\phi(\eta)+\frac{1}{2}|\eta|^{2}}dm(\eta)|e^{-\frac{1}{2}|z|^{2}}\bigg\}\\&
 +\sup_{z\in \mathbb{C}}\big\{|f(0)\big(\int_{0}^{z}\delta_{0}^{-2}A(\zeta)d\zeta\big)|e^{-\frac{1}{2}|z|^{2}}\big\}+|f'(0)|
 \end{aligned}
\end{equation*}
Thus $$
\|f'\|_{\infty}\bigg(1-\sup_{z\in \mathbb{C}}\bigg\{|\int_{\mathbb{C}}\big(\int_{0}^{z}\overline
{K'_{\zeta}(\eta)}A(\zeta)d\zeta\big)(1+\phi'(\eta))^{-2}e^{-2\phi(\eta)+\frac{1}{2}|\eta|^{2}}dm(\eta)|e^{-\frac{1}{2}|z|^{2}}\bigg\}\bigg)
 $$$$\leq\sup_{z\in \mathbb{C}}\big\{|f(0)\big(\int_{0}^{z}\delta_{0}^{-2}A(\zeta)d\zeta\big)|e^{-\frac{1}{2}|z|^{2}}\big\}+|f'(0)|
$$
If $\|f'\|_{\infty}=\infty$, it  contradicts  the condition of Theorem $1.7$. Therefore, $$
\|f'\|_{\infty}\leq\frac{1}{1-Y_{K}(A)}\bigg(\sup_{z\in \mathbb{C}}\big\{|f(0)\big(\int_{0}^{z}\delta_{0}^{-2}A(\zeta)d\zeta\big)|e^{-\frac{1}{2}|z|^{2}}\big\}+|f'(0)|\bigg)<\infty
$$ and $f'\in F^{\infty}.$
  \end{proof}
\section{Proof of Theorems $1.8$ }
  \begin{proof} [\bf Proof of Theorem $1.8$]  By  the proof of Theorem $1.6$, we get   \begin{equation*}
\begin{aligned}
 f'(z)& =-\int_{\mathbb{C}} \lim_{n\rightarrow\infty}f_{n}(\eta)e^{-2\phi(\eta)}\big(\int_{0}^{z}\overline{K_{\zeta}(\eta)}A(\zeta)d\zeta\big)dm(\eta)+f'(0)\\&
 =-\int_{\mathbb{C}}f(\eta)e^{-2\phi(\eta)}\big(\int_{0}^{z}\overline{K_{\zeta}(\eta)}A(\zeta)d\zeta\big)dm(\eta)+f'(0),\quad z\in\mathbb{C}.
 \end{aligned}
\end{equation*}
Since the condition of Theorem $1.8$ satisfies Lemma $2.8$, we get  $$ \|f\|^{2}_{F_{\phi}^{2}}\leq C\bigg(|f(0)|^{2}+\int_{\mathbb{C}}|f'(z)|\frac{e^{-2\phi(|z|)}}{(1+\phi'(|z|))^{2}}dm(z)\bigg)$$
$$\leq C\bigg(|f(0)|^{2}+\int_{\mathbb{C}}|\int_{\mathbb{C}}f(\eta)e^{-2\phi(\eta)}\big(\int_{0}^{z}\overline{K_{\zeta}(\eta)}A(\zeta)d\zeta\big)dm(\eta)|
\frac{e^{-2\phi(|z|)}}{(1+\phi'(|z|))^{2}}dm(z)+
$$
$$|f'(0)|\int_{\mathbb{C}}\frac{e^{-2\phi(|z|)}}{(1+\phi'(|z|))^{2}}dm(z)\bigg),\quad z\in\mathbb{C}.
$$
 By Lemma $2.8$, we get  $$\int_{\mathbb{C}}\frac{e^{-2\phi(|z|)}}{(1+\phi'(|z|))^{2}}dm(z)\leq  C \|z\|^{2}_{F_{\phi}^{2}}.$$ Using Lemma $2.6$, it is easy to get that there exists a positive number $M$ such that $ \|z\|^{2}_{F_{\phi}^{2}}<M$. Therefore,  $$\|f\|^{2}_{F_{\phi}^{2}}\leq P+ \|f\|^{2}_{F_{\phi}^{2}}\sup_{\eta\in\mathbb{C}}\int_{\mathbb{C}}\big(\int_{0}^{z}\overline{K_{\zeta}(\eta)}A(\zeta)d\zeta\big)\frac{e^{-2\phi(|z|)}}{(1+\phi'(|z|))^{2}}dm(z)$$
where $P=C(|f(0)|^{2}+|f'(0)|CM)$. Then $$\|f\|^{2}_{F_{\phi}^{2}}\bigg(1- \sup_{\eta\in\mathbb{C}}\int_{\mathbb{C}}\big(\int_{0}^{z}\overline{K_{\zeta}(\eta)}
A(\zeta)d\zeta\big)\frac{e^{-2\phi(|z|)}}{(1+\phi'(|z|))^{2}}dm(z)\bigg)\leq P.$$

If $\|f\|^{2}_{F_{\phi}^{2}}=\infty$, it  contradicts  the condition of Theorem $1.8$. Therefore, $$
\|f\|^{2}_{F_{\phi}^{2}}<\frac{P}{1-Z_{K}(A)}<\infty
$$ and $f\in F_{\phi}^{2}.$
\end{proof}

\section{Acknowledgements}
The first author would like to thank Department of Physics and Mathematics,  University of Eastern Finland, for providing a good environment during the preparation of this work. We also thank the referees for many important and useful comments.

\end{document}